\documentclass[10pt]{article}
\usepackage{amsmath,amssymb,amsthm}
\usepackage{epsfig}
\usepackage{color}

\title{A solution to an open problem on lower against number in graphs}

\author {
Babak Samadi\\
Department of Mathematics\\
Arak University\\
Arak, IRI\\
{\tt samadibabak62@gmail.com}\vspace{3mm}
}

\date{}

 \newtheorem{theorem}{Theorem}[section]

\newtheorem{lemma}[theorem]{Lemma}

\theoremstyle{definition}

\begin{document}

\maketitle

\begin{abstract}
\noindent  In \cite{w} the problem of finding a sharp lower bound on lower against number of a general graph is mentioned as an open question. We solve the problem by establishing a tight lower bound on lower against number of a general graph in terms of order and maximum degree.  \vspace{3mm}\\
{\bf Keywords:} Lower against number, maximal negative function.
\end{abstract}

\section{Introduction}
Throughout this paper, let $G$ be a finite connected graph with vertex set $V=V(G)$ and edge set $E=E(G)$. We use \cite{we} for terminology and notations which are not defined here. The open neighborhood of a vertex $v$ is denoted by $N(v)$, and the closed neighborhood of $v$ is $N[v]=N(v)\cup \{v\}$. For a subset $S‎\subseteq V(G)‎$, $N(S)=‎\cup‎_{v‎\in S‎}N(v)‎‎$. A graph $G$ is called $r$-regular if $deg(v)=r$ for every $v‎\in V(G)‎$, and nearly $r$-regular if $deg(v)‎\in\{r-1,r\}‎$ for every $v‎\in V(G)‎$.\\
Let $S \subseteq V$. For a real-valued function $f: V\rightarrow R$ we define $ f(S)=\sum_{v \in S }f(v)$. Also, $f(V)$ is the weight of $f$. A function $f:V\rightarrow \{ -1, 1\}$ is called negative if $f(N[v])‎\leq 1‎$, for every $v‎\in V(G)‎$. The maximum of values of $f(V(G))$, taken over all negative functions $f$, is called the against number $‎\beta‎_{N}(G)‎‎$. The author in \cite{w} exhibited a real-world application of it to social networks. (This concept was introduced by Zelinka \cite{z} as signed $2$-independence number).\\
  A negative function $f$ of a graph $G$ is maximal if there exist no negative function $g$ such that $g‎\neq f‎$ and $g(v)‎\geq f(v)‎$ for every $v‎\in V(G)‎$. The minimum of values of $f(V(G))$, taken over all maximal negative functions $f$, is called the lower against number and is denoted by $\beta‎_{N}^*(G)‎‎$.\\
In \cite{w}, Wang proved the following lower bounds on $\beta‎_{N}^*(G)‎‎$ for  regular and nearly regular graphs.
\begin{theorem}
Let $G$ is an $r$-regular graph of order $n$. Then, $\beta‎_{N}^*(G)‎\geq (r+2-r^2)n/(r+2+r^2)‎$ for $r$ even, and $\beta‎_{N}^*(G)‎\geq (1-r)n/(1+r)$ for $r$ odd. This bound is best possible.
\end{theorem}
\begin{theorem}
For any nearly $r$-regular graph $G$ of order $n$, $\beta‎_{N}^*(G)‎\geq (1-r)n/(1+r)‎$. Furthermore, this bound is sharp.
\end{theorem}
Also, the author posed the following question as an open problem:
{\em What is a sharp lower bound on $\beta‎_{N}^*(G)$ for a general graph $G$?}\\
Recently, Zhao in \cite{zh} proved that if $G$ is a graph of order $n$ with minimum degree $‎\delta‎$ and maximum degree $‎\Delta‎$     , then
$$\beta‎_{N}^*(G)‎‎‎\geq ‎(‎\delta‎+2+‎\delta ‎\Delta-2‎\Delta^2‎‎‎)n)/(‎\delta+2-‎\delta ‎\Delta+2‎\Delta^2‎‎‎‎)$$
for $‎\delta‎$ even, and
$$\beta‎_{N}^*(G)‎‎‎\geq (‎\delta+1-‎\Delta+‎\delta ‎\Delta-2‎\Delta^2‎‎‎‎‎)n/(‎\delta+1+‎\Delta-‎\delta ‎\Delta+2‎\Delta^2‎‎‎‎‎)$$
for $‎\delta‎$ odd. Moreover he showed that these bounds are sharp.\\
In this paper, in answer to the question, we give a sharp lower bound on the lower against number of a general graph just in terms of order and maximum degree that is tighter that ones in \cite{zh}. Also, we conclude Theorem 1.1 and Theorem 1.2 as immediate results of our main theorem. 

\section{A lower bound on $‎\beta‎_{N}^*(G)‎‎$}
We need the following lemma.
\begin{lemma} \cite{w}
A negative function $f$ of a graph $G$ is maximal if and only if for every $v‎\in V(G)‎$ with $f(v)=-1$, there exists at least one vertex $u‎\in N[v]‎$ such that $f(N[v])=0$ or $1$.
\end{lemma}
We are now in a position to present the main result of this paper.
\begin{theorem}
Let $G$ be a graph of order $n$ with maximun degree $‎\Delta‎$. Then 
$$\beta‎_{N}^*(G)‎ ‎\geq‎ \left\{
\begin{array}{ccc}
(\dfrac{1-‎\Delta‎}{1+‎\Delta‎})n& \text{ } & ‎\Delta‎\geq ‎\delta+1 \ \ or\ \ ‎\delta=‎\Delta‎\equiv‎‎‎1\ (mod \ 2)‎‎‎‎‎\\ \vspace{0mm}\\
(‎\dfrac{‎\Delta‎+2-‎\Delta‎^2}{‎\Delta‎+2+‎\Delta‎^2}‎)n & \text{ } & otherwise.‎‎\\
\end{array} \right.$$
and these bounds are sharp.
\end{theorem}
\begin{proof}
If $‎\delta=‎\Delta‎\equiv0\ (mod \ 2)‎‎‎$, then desired result follows by Theorem 1.1. Hence in what follows we may assume that $‎\Delta‎\geq ‎\delta+1‎‎‎$ or $\delta=‎\Delta‎\equiv‎‎‎1\ (mod \ 2)‎‎‎‎‎$.\\
Let $f$ be a maximal negative function of $G$ with weight $f(V(G))=\beta‎_{N}^*(G)$ and $M=\{v‎\in V|f(v)=-1‎\}$ and $P=\{v‎\in V|f(v)=1‎\}$. Also, $m=|M|$ and $p=|P|$. For notational convenience, we set $l=‎\lfloor‎\frac{‎\Delta‎}{2}‎‎\rfloor+1‎‎$ and $k=‎\lfloor‎\frac{‎\delta‎}{2}‎‎\rfloor+1$. We define $A‎_{i}=\{v‎\in M||N(v)‎\cap P‎|=i‎\}‎$ and $a‎_{i}‎=|A‎_{i}‎|$, for all $0‎\leq i‎\leq ‎l$.
Let $v‎\in M‎$. Since $f$ is a negative function, then $v$ has at most $l$ neighbors in $P$. Therefore, $P$ is the disjoint union, for $0‎\leq i‎\leq ‎l‎‎‎‎$, of the sets $A‎_{i}‎$. Now we get 
\begin{equation}
n=p+m=p+\sum_{i=0}^{l‎‎‎‎}a‎_{i}‎.
\end{equation}
On the other hand, if $[M,P]$ is the set of edges having one end point in $M$ and the other in $P$, then
\begin{equation}
|[M,P]|=\sum_{i=1}^{l‎‎‎‎}ia‎_{i}‎\leq p‎\Delta.‎‎
\end{equation}
{\bf Case 1.} If $A‎_{0}=‎‎‎\phi‎$. By inequalities (1) and (2), we have
 $$n=p+\sum_{i=1}^{l‎‎‎‎}a‎_{i}‎\leq p+\sum_{i=1}^{l}ia‎_{i}‎\leq p+p‎\Delta‎‎‎.$$
Therefore, $p=(n+\beta‎_{N}^*(G))/2‎\geq ‎\dfrac{n}{1+‎\Delta‎}‎‎$, which implies the desired lower bound.\\
{\bf Case 2.} If $A‎_{0}‎‎\neq ‎\phi$. Let $v‎\in A‎_{0}‎$. Obviously, $f(N[v])‎\leq-2‎$. Now Lemma 1 implies that there exists a vertex $u‎\in N[v]‎$ such that $f(N[u])=0$ or $1$. This shows that the set $Q=\{v‎\in N(A‎_{0}‎)|f(N[v])=0\  or\  1‎\}$ is nonempty. Let $v‎\in ‎\cup‎‎‎_{i=0}^{k-1}‎‎‎‎‎‎‎‎A‎_{i}‎$. Then
\begin{equation*}
\begin{array}{lcl}
f(N[v])&=&|N[v]‎\cap P‎|-|N[v]‎\cap (V‎\setminus P‎)‎|=2|N[v]‎\cap P‎|-|N[v]|\\
&‎\leq‎&2(k-1) -deg(v)-1‎\leq -1.‎
\end{array}
\end{equation*}
Therefore $v$ does not belong to $Q$. Hence, $Q‎\subseteq \cup‎‎‎_{i=‎k}^{l}‎‎‎‎‎‎‎‎A‎_{i}‎‎$. Suppose that $u‎\in Q‎\cap A‎_{i}‎‎‎$, for $k‎\leq i‎\leq l‎‎$. We claim that $|N(u)‎\cap A‎_{o}‎‎‎|\leq i-1‎$. Suppose to the contrary that $|N(u)‎\cap A‎_{o}|‎\geq i‎$. Then
\begin{equation*}
\begin{array}{lcl}
0\  or\ 1=f(N[u])&=&-1+|N(u)‎\cap P‎|-|N(u)‎\cap M‎|\\
 ‎&\leq & -1+i-|N(u)‎\cap A‎_{0}‎‎|‎‎\leq-1‎
\end{array}
\end{equation*}
a contradiction. Thus $Q‎\cap A‎_{i}‎‎$ has at most $(i-1)|Q‎\cap A‎_{i}|$ neighbors in $A‎_{0}‎$. Since $f$ is a maximal negative function, for every vertex $v‎\in A‎_{0}‎‎$ there exists a vertex $u‎\in Q‎$ such that $u‎\neq v‎$, which implies $u‎\in Q‎\cap A‎_{i}‎‎‎$, for some $k‎\leq i‎\leq l‎‎$. Hence $A‎_{0}‎‎\subseteq ‎\cup‎‎‎_{i=k}^{l}N(Q‎\cap A‎_{i}‎‎‎)$. Now we deduce that
$$a‎_{0}=|A‎_{0}|‎\leq \sum_{i=k}^{l‎‎‎‎}|N(Q‎\cap A‎_{i}‎‎‎)‎\cap A‎_{0}‎‎|‎\leq ‎‎‎\sum_{i=k}^{l‎‎‎‎}|Q‎\cap A‎_{i}|(i-1)‎\leq \sum_{i=k}^{l‎‎‎‎}(i-1)‎a‎_{i}‎.$$
By (1), we have 
$$n=p+a‎_{0}‎+\sum_{i=1}^{l‎‎‎‎}a‎_{i}‎‎\leq p+‎\sum_{i=k}^{l‎‎‎‎}(i-1)‎a‎_{i}+\sum_{i=1}^{k-1‎‎‎‎}‎a‎_{i}+\sum_{i=k}^{l‎‎‎‎}‎a‎_{i}.$$
Thus
 $$n‎\leq p+\sum_{i=1}^{l‎‎‎‎}ia‎_{i}‎\leq p+p‎\Delta‎‎‎.$$
 Therefore $p‎=(n+\beta‎_{N}^*(G))/2\geq ‎\dfrac{n}{‎\Delta+1‎}‎‎$, as desired.\\
 Since Theorem 1.2 (also Theorem 1.1) is a special case of this theorem, we see that this lower bound is sharp.
\end{proof}
Comparing Theorem 2.2 with its corresponding result in \cite{zh} we can see that the lower bounds in Theorem 2.2 are tighter that their corresponding ones in \cite{zh}. Moreover, Theorem 1.1 and Theorem 1.2 are immediate results of Theorem 2.2 when $‎\delta=‎\Delta=r‎‎$.



\begin{thebibliography}{9}

\bibitem{w} C. Wang, {\em Voting ‎against in regular and nearly regular graphs‎}, Applicable Analysis and Discrete Mathematics, {\bf 4} (2010), 207--218.
\bibitem{we} D. B. West, Introduction to graph theory (Second Edition), Prentice Hall, USA, 2001.
\bibitem{z} B. Zelinka, {\em On signed 2-independence number of graphs}, manuscript.
\bibitem{zh} W. Zhao, {\em Lower against number in graphs}, International Journal of Pure and Applied Mathematics, {\bf 77} (2012), 149--155.
\end{thebibliography}
\end{document}